\documentclass[11pt]{amsart}
\usepackage{amsmath}
\usepackage{amsfonts}
\usepackage{amsthm}
\usepackage{amssymb}
\usepackage{mathrsfs}
\usepackage{latexsym}
\usepackage{verbatim}
\usepackage{diagrams}
\usepackage{a4wide}
\usepackage{enumerate}
\usepackage{hyperref}
\usepackage{amscd}
\usepackage{ stmaryrd }
\usepackage{hyperref}

\def\ad{\mathrm{ad}}
\def\st{\mathrm{st}}
\def\ur{\mathrm{ur}}
\def\Rloc{R^{\mathrm{loc}}}
\def\Htw{\widetilde{H}}
\def\t{\mathbf{t}}

\newtheorem{theoremOrd}{Theorem}

\renewcommand{\thetheoremName}

\newtheorem{conjectureOrd}[theoremOrd]{Conjecture}

\def\Adele{\mathbb{A}}
\def\Afinite{\mathbb{A}^{(\infty)}}
\def\LL{\mathcal{L}}
\def\Aut{\mathrm{Aut}}
\def\Sym{\mathrm{Sym}}
\def\codim{\mathrm{codim}}
\def\Res{\mathrm{Res}}
\def\H{\mathcal{H}}
\def\T{\mathbf{T}}
\def\Z{\mathbf{Z}}
\def\G{\mathbb{G}}
\def\A{\mathbb{A}}
\def\OL{\mathcal{O}}
\def\Q{\mathbf{Q}}
\def\F{\mathbf{F}}
\def\R{\mathbf{R}}
\def\C{\mathbf{C}}
\def\a{\mathfrak{a}}
\def\m{\mathfrak{m}}

\def\Qbar{\overline{\Q}}
\def\rhobar{\overline{\rho}}
\def\Tor{\mathrm{Tor}}
\def\PGL{\mathrm{PGL}}
\def\Ann{\mathrm{Ann}}
\def\GL{\mathrm{GL}}
\def\Gal{\mathrm{Gal}}
\def\Frob{\mathrm{Frob}}

\def\eps{\epsilon}
\def\NF{N_{F/\Q}}
\def\Rloc{R_{\mathrm{loc}}}
\DeclareMathOperator\red{red}

\newtheorem{theorem}{Theorem}[section]

\newtheorem{df}[theorem]{Definition}
\newtheorem{lemma}[theorem]{Lemma}
\newtheorem{prop}[theorem]{Proposition}

\newtheorem{remark}[theorem]{Remark}

\begin{document}

\title{Semistable Modularity Lifting over imaginary quadratic fields}
\author{Frank Calegari} 
\thanks{The first author was supported in part by NSF  Grant
  DMS-1404620.}
  
\maketitle


\section{Introduction}

In this paper,  we prove  non-minimal  modularity lifting theorem  for ordinary
Galois representations over imaginary quadratic fields. Our first theorem is as follows:

\begin{theorem} \label{theorem:elliptic} Let $F/\Q$ be an imaginary quadratic field, and let $p > 2$ be a prime
which is unramified in~$F$. 
Let $E/F$ be a semistable elliptic curve with ordinary reduction at all $v|p$. Suppose that
the mod~$p$ Galois representation:
$$\rhobar_{E,p}:G_{F} \rightarrow \Aut(E[p]) = \GL_2(\F_p)$$
is absolutely irreducible over $G_{F(\zeta_p)}$  and is modular. 
Assume that the Galois representations attached to ordinary cohomology
classes for Bianchi groups are ordinary  --- see Conjecture~\ref{conj:Ord}.
Then~$E$ is modular.
\end{theorem}

The modularity hypotheses is satisfied, for example, when $\rhobar_{E,p}$ extends to an
odd representation~$\rhobar$ of $G_{\Q}$. In particular, if~$3$ or~$5$ is unramified in~$F$,
this theorem implies --- conditionally on Conjecture~\ref{conj:Ord} --- the modularity of infinitely many~$j$ invariants in~$F \setminus \Q$, because
one can take any $E/F$ such that $E[p] \simeq A[p]$ where $A/F$ is the base change of an
elliptic curve over~$\Q$.
If~$p = 3$, then the representation associated to~$E[3]$ has solvable image 
($\PGL_2(\F_3) \simeq S_4$). However, unlike in the case of totally real fields, the automorphic
form~$\pi$ associated to the corresponding Artin representation does not in any obvious way admit ``congruences'' to modular forms of cohomological weight,
and hence the modularity hypothesis cannot be deduced from  
Langlands--Tunnell~\cite{Langlands,Tunnell} (as in the deduction of Theorem~0.3 of~\cite{W} 
from Theorem~0.2). We deduce Theorem~\ref{theorem:elliptic} from the following:

\begin{theorem} \label{theorem:main} Assume conjecture~\ref{conj:Ord}. 
Suppose that $p > 2$ is unramified in~$F$, and let
$$\rho: G_{F} \rightarrow \GL_2(\Qbar_p)$$
be a continuous irreducible Galois representation unramified
outside finitely many primes. Assume that:
\begin{enumerate}
\item The determinant of~$\rho$ is~$\eps^{k-1}$, where~$\eps$ is the cyclotomic character.
\item  If $v|p$, then $\rho|D_v$ is ordinary and crystalline with Hodge--Tate weights $[0,k-1]$
for some $k \ge 2$. 
\item $\rhobar |_{F(\zeta_p)}$ is absolutely irreducible. If~$p = 5$, then the projective image of~$\rhobar$
is not~$\PGL_2(\F_5)$.
\item $\rhobar$ is modular.
\item  If $\rho$ is ramified at $v \nmid p$, then 
$\rho|D_v$ is semistable, that is,
$\rho|I_v$ is unipotent.
\end{enumerate}
Then $\rho$ is modular.
\end{theorem}

\medskip

The main idea of this paper is to combine the modularity lifting theorems of~\cite{CG} with
the techniques on level raising developed in~\cite{CV}. Wiles' original argument for
proving modularity in the non-minimal case required two ingredients: the use of a subtle numerical
criterion concerning complete intersections which were finite flat over a ring of integers~$\OL$,
and Ihara's Lemma.
Although Ihara's Lemma (in some form) is available
for imaginary quadratic fields (see~\cite{CV}, {Ch~IV}), it seems tricky to generalize the
numerical criterion to this setting --- the Hecke rings are invariably not torsion free, and
are rarely complete intersections even in the minimal case (the arguments of~\cite{CG}
naturally present the minimal deformation ring as a quotient of a power series in $q-1$ variables
by~$q$ elements). Instead, the idea is to work in a context in which the ``minimal''
deformations are all Steinberg at a collection of auxiliary primes~$S$. It turns out that a natural
setting where one expects this to be true is in the cohomology of the~$S$-arithmetic group
$\PGL_2(\OL_F[1/S])$. In order to apply the methods of~\cite{CG}, one requires two main
auxiliary hypotheses to hold. The first is that the range of cohomology which doesn't vanish after
localizing at a suitable maximal ideal~$\m$ has length~$\ell_0 = 1$. When the number of primes~$m$
dividing~$S$ is zero, this is an easy lemma, and was already noted in~\cite{CV} (Lemma~5.9). When~$m = 1$, the required vanishing follows from the congruence subgroup
property of $\PGL_2(\OL_F[1/S])$ as proved by Serre~\cite{CSP}. When $m > 1$, however, the problem
is more subtle. The cohomology in this range may well be non-trivial and is related
to classes arising from the algebraic~$K$-theory of~$\OL_F$ (as explained in~\cite{CV}). 
Nevertheless, 
if one  first completes at a non-Eisenstein maximal ideal~$\m$,
the necessary vanishing required for applications to modularity is expected to hold,
and indeed was 
conjectured 
in~\cite{CV}.
We do not, however, prove this vanishing conjecture in this paper. Instead, we 
prove that the patched cohomology in these lower degrees is sufficiently
small (as a module over the patched diamond operator ring~$S_{\infty}
= \OL \llbracket x_1, \ldots ,x_q\rrbracket$) that a modified version of the argument of~\cite{CG}
still applies.

\medskip

There are three further technical obstacles which must be dealt with. We now discuss them in turn.

The methods of~\cite{CG} require that the Galois representations
(constructed in much greater generality than used here by~\cite{Scholze})
satisfy the expected local properties at~$v|p$ and~$v \nmid p$.
The required local--global
compatibility for~$v \nmid p$ was established by~\cite{Varma}. The required local--global
compatibility for~$v | p$ in the ordinary case is still open. 
We do not resolve this issue here, but instead make the weakest possible
assumption necessary for applications --- 
namely that cohomology classes 
on which 
the operator~$U_v$ is invertible give rise to Galois representations which admit an unramified
quotient on which Frobenius at~$v$ acts by~$U_v$. 
We believe that this formulation 
(Conjecture~\ref{conj:A}) might be amenable to current technology.

A second issue that we must deal with is relating the modularity assumption on~$\rhobar$
for $\PGL_2(\OL_F)$ to the required modularity for the group~$\PGL_2(\OL_F[1/S])$.
This is a form of level raising, and to prove it we use the level raising spectral sequence of~\cite{CV}. This
part of the paper is not conditional on any conjectures, and may be viewed as a generalization
of Ribet's level raising theorem in this context. Many of the ideas here are already present in~\cite{CV}.

The final issue which must be addressed is that Scholze's Galois representations are only defined over the ring~$\T/I$ for some nilpotent ideal~$I$ with a fixed level of nilpotence (depending on the group). Moreover, some of the constructions here also require increasing the degree
of nilpotence. Thus we are also required to explain how the methods of~\cite{CG} may be adapted to
this context. This last point requires only a technical modification. The essential point is that
if a finitely generated~$S_{\infty}= \OL \llbracket x_1, \ldots ,x_{q}\rrbracket$-module~$M$ is annhilated by~$I^{2}$, then~$M/I$
has the same co-dimension over~$S_{\infty}$ as~$M$.

\begin{remark} \emph{Our theorem and its proof may be generalized to allow other ramification
types at auxiliary primes~$v \nmid p$,  \emph{providing} that this new ramification is of minimal type,
e.g~$\rho(I_v) \simeq \rhobar(I_v)$. This can presumably be achieved using the modification
found by Diamond~\cite{DiamondVexing} and also developed in~\cite{CDT}. The required change
would be to modify the corresponding local system at such primes. We avoid this in order to
clarify exactly the innovative aspects of this paper.}
\end{remark}

Suppose that~$\rhobar$ satisfies the conditions of Theorem~\ref{theorem:main}.
The assumption that~$\rhobar$ is modular is defined to mean that
 the localization $H_1(Y,\LL)_{\m}  \ne 0$
for a certain arithmetic quotient~$Y$ and a local system~$\LL$ corresponding to~$\rhobar$ and
maximal ideal~$\m$ of the corresponding anaemic Hecke algebra. (This is a weaker property
than requiring~$\rhobar$ to be the mod-$p$ reduction of a representation associated to
an automorphic form of minimal level.)
This is equivalent to asking that~$H_1(Y,\LL/\varpi)_{\m}$ is non-zero and also to
asking that~$H_2(Y,\LL/\varpi)_{\m}$ is non-zero. (If~$H_2$ vanishes, then~$H_1$ is torsion free,
which implies that there exists a corresponding automorphic form, which then must contribute to~$H_2$.)

\subsection{Notation}

We fix an imaginary quadratic field $F/\Q$, and an odd prime~$p$ which is unramified in~$F$.
Let~$\OL$ denote the ring of integers in a finite extension of~$\Z_p$. We shall assume that~$\OL$
is sufficiently large that it admits inclusions~$\OL_{F,v} \rightarrow \OL$ for each~$v|p$, and that the residue
field~$k = \OL/\varpi$ contains sufficiently many eigenvalues of any relevant representation~$\rhobar$. 
Let~$N$ denote a tame level prime to~$p$. Let~$S$ denote a finite set of primes disjoint
from~$N$ and~$p$. Let~$m$ denote the number of primes in~$S$. By abuse of notation, we sometimes
use~$S$ to denote the ideal of~$\OL_F$ which is the product of the primes in~$S$.

Let $\G = \Res_{F/\Q}(\PGL(2)/F)$,  and write $G_{\infty} = \G( \R) = \PGL_2(\C)$.
Let $K_{\infty}$ denote a maximal compact of $G_{\infty}$ with
connected component $K^{0}_{\infty}$, so~$G_{\infty}/K^{0}_{\infty} = \H$ is hyperbolic~$3$-space.
Let $\Adele$ be the adele ring of $\Q$,  and
$\Afinite$ the finite  adeles.  
For any compact open subgroup $K$ of
$\G(\Afinite)$, we may define an ``arithmetic manifold'' (or rather ``arithmetic orbifold'')
$Y(K)$ as follows:
$$Y(K):=  \G(F) \backslash (\H \times \G(\Afinite))/K = \G(F) \backslash \G(\Adele)/K^0_{\infty} K.$$ 
The orbifold $Y(K)$ is not compact but has finite volume; it may be disconnected.

\medskip

Let~$K_0(v)$ denote the Iwahori subgroup of~$\PGL_2(\OL_{F,v})$, and let~$K_1(v)$ denote the pro-$v$ Iwahori,
which is the kernel of the map~$K_0(v) \rightarrow k^{\times}_v$.

\begin{df} Let~$R$ be an ideal of~$\OL_{F}$.
If we choose~$K$ to consist of the level structure~$K_0(v)$ for~$v|R$ and maximal level structure elsewhere,
then we write~$Y_0(R)$ for~$Y(K)$. If~$K$ has level~$K_0(v)$ for~$v|R$ and~$K_1(v)$ for~$v|Q$ for some auxiliary~$Q$,
we write~$Y_1(Q;R)$ for~$Y(K)$.
\end{df}

Given~$S$, we may similarly define~$S$-arithmetic locally symmetric spaces (directly following~\S3.6 and~\S4.4  of~\cite{CV}) as follows.
Let $\mathscr{B}_S$ be the product of the Bruhat--Tits buildings of $\PGL_2(F_v)$  for $v \in S$;
we regard each building as a contractible simplicial complex, and so $\mathscr{B}_S$ is a contractible square complex. 
In particular, $\mathscr{B}_S$ has a natural filtration:
$$\mathscr{B}_S^0 \subset \mathscr{B}_S^1 \subset \mathscr{B}_S^2 \subset \dots$$
where $\mathscr{B}_S^{(j)}$ comprises the union of cells of dimension $\leq j$. 
Consider the quotient 
$$Y(K[\frac{1}{S}])  :=  \G(F) \backslash \left( G_{\infty}/K_{\infty} \times \mathscr{B}_S \times \G(\A^{\infty,S})/ K^{S} \right).$$
This has a natural filtration by  spaces $Y_S^j$ defined by replacing $\mathscr{B}_S$
with $\mathscr{B}_S^j$. 
The space $Y_S^{j} - Y_S^{j-1}$ is a smooth manifold of dimension
$\dim(Y_{\{\infty\}})+ j$.  
 When~$K$ has type~$K_0(v)$ for~$v|R$, we write
$Y_0(R)[1/S]$  for these spaces, and, with additional level~$K_1(v)$ for~$v|Q$ and~$Q$ prime to~$R$ and~$S$,
we write~$Y_1(Q;R)[1/S]$. The cohomology of~$Y[1/S]$ and its covers will naturally recover spaces of automorphic
forms which are Steinberg at primes dividing~$S$. In order to deal with
representations which correspond to a quadratic unramified twist of the Steinberg representation, we need to introduce
a local system as follows.

\medskip

Let $\epsilon: S \rightarrow \{\pm 1\}$ be a choice of sign for every place $v \in S$. 
Associated to $\epsilon$ there is a natural character  $\chi_{\epsilon}  : \G(F) \rightarrow \{\pm 1\}$,
namely $\prod_{v \in S : \epsilon(v) = -1} \chi_v$; 
here $\chi_v$ is the
``parity of the valuation of determinant,'' 
obtained via the natural maps $$\G(F) \stackrel{\det}{\longrightarrow} F^{\times}/(F^{\times})^2 \rightarrow \prod_{v} F_v^{\times}/(F_v^{\times})^2 \  \stackrel{v}{\longrightarrow} \pm 1,$$
where the final map is the parity of the valuation.  Correspondingly, we obtain a {\em sheaf of $\OL$-modules}, denoted $\mathcal{F}_{\epsilon}$,  on the space $Y[1/S]$. Namely, the total space of the local system $\mathcal{F}_{\epsilon}$
corresponds to the quotient  of 
$$\left( G_{\infty}/K_{\infty} \times \mathscr{B}_S \times \G(\A^{\infty,S})) / K^S \right)$$
 by the action of $\G(F)$: the natural action on the first factor,
and the action via $\chi_\epsilon$ on the second factor.  Finally, let~$\mathcal{F}$ be the direct sum of~$\mathcal{F}_{\eps}$ over
all~$2^m = 2^{|S|}$ choices of sign~$\eps$.

\subsection{Local Systems}
For a pair~$(m,n)$ of integers at least two, one has the representation
$$\Sym^{m-2} \C^2 \otimes \overline{\Sym^{n-2} \C^2}$$
of~$\GL_2(\C)$. These representations give rise to local systems of~$Y[1/S]$ (and its covers) defined over~$\OL_F[1/S]$, and hence also 
to~$\OL$. Similarly, for any~$S$ and any~$\epsilon$ as above, there are corresponding local systems~$\LL$ obtained by tensoring
this local system with~$\mathcal{F}$.  

\begin{remark}[Amalgams] \label{remark:amalgam}
\emph{The structure of the groups~$\PGL_2(\OL_F[1/S])$ and its congruence subgroups for~$S = T \cup \{v\}$ as amalgam 
of~$\PGL_2(\OL_F[1/T])$ with itself over the Iwahori subgroup of level~$v$ implies, by the long exact sequence associated
to an amalgam, that there is an exact sequence:
{\small
$$\ldots  \rightarrow H_{n}(Y_1(Q;R)[1/T],\LL/\varpi^r)^2 \rightarrow H_n(Y_1(Q_N;R)[1/S],\LL/\varpi^r)
\rightarrow H_{n-1}(Y_1(Q_N;Rv)[1/T],\LL/\varpi^r)  \rightarrow  \ldots$$
}
This simple relationship between~$S$ arithmetic groups is special to the case~$n = 2$, and is crucial for our
inductive arguments.
}
\end{remark}

\begin{remark}[Orbifold Cohomology] 
\emph{ Whenever we write~$H_*(Y,\LL)$ for an orbifold~$Y$, we mean the cohomology as orbifold cohomology rather than the cohomology
of the underlying space. }
\end{remark}

\subsection{Hecke Operators}

We may define Hecke operators~$T_v$ for primes~$v$ not dividing~$S$ acting on~$H_*(Y_1(Q;R)[1/S], \LL)$ in the usual way.
For primes~$v|S$, one also has the operators~$U_v$.
The action of~$U_v$ on the cohomology of~$\mathcal{F}_{\eps}$ is by~$U_v = \eps(v) \in \{ \pm 1 \}$. More generally,  on~$H_*(Y_1(Q;R)[1/S],\LL)$,
we have (cf. the proof of Lemma~9.5 of~\cite{CG}):
$$U^2_{v} = U_{v^2} = 1.$$
For primes~$v|RQ$, there is also a Hecke operator we denote by~$U_v$. 
We denote by $\T_Q$ be the~$\OL$-algebra of endomorphisms generated by the action
of these Hecke operators on the direct sum of cohomology groups~$H_*(Y_1(Q;R)[1/S],\LL/\varpi^r)$ for any given~$m$, and let~$\m$ be a maximal ideal of~$\T$.

\section{Galois Representations}

Suppose that~$\LL$ has parallel weight~$(k,k)$ for some integer~$k \ge 2$. 
Our main assumption on the existence of Galois representations is as follows:

\begin{conjectureOrd}[Ordinary $\Rightarrow$ Ordinary]
\label{conj:A}
\label{conj:Ord} Assume that $\m$ is non-Eisenstein of residue characteristic~$p > 2$
and is associated to a Galois representation $\rhobar$,  and assume that~$T_v \notin \m$ for~$v|p$.
 Then there exists a continuous Galois representation
$\rho = \rho_{\m}:G_F \rightarrow \GL_2(\T_{Q,\m})$ with the following properties:
\begin{enumerate}
\item\label{char-poly} If $\lambda\not\in R \cup Q\cup\{v|p\} \cup S$ is a prime of $F$, then $\rho$ is unramified
at $\lambda$, and the characteristic polynomial of $\rho(\Frob_{\lambda})$ is
$$Y^2 - T_{\lambda} X + {\NF}(\lambda)^{k-1} \in \T_{Q,\m}[X].$$
\item For~$v|p$, the representation~$\rho | D_v$ is ordinary with eigenvalue the unit root of~$X^2 - T_v X + N(v)^{k-1}$.
\item If $v \in R$, then $\rho | I_v$ is unipotent.
\item If $v \in S$, then $\rho | I_v$ is unipotent, and moreover the characteristic polynomial of (any) lift of Frobenius
is
$$X^2 - U_v (N(v)^{k-1} + N(v)^k) + N(v)^{2k-1}.$$
\item\label{lgc-Q} If $v\in Q$,  the operators $T_{\alpha}$ for $\alpha \in F^{\times}_v \subset \A^{\infty,\times}_F$
are invertible. Let $\phi$ denote the character of $D_v = \Gal(\overline{F}_v/F_v)$ which, by class field theory,
is associated to the resulting homomorphism:
$$F^{\times}_v \rightarrow \T^{\times}_{Q,\m}$$
given by sending $x$ to $T_x$. By assumption, the image of $\phi \mod \m$ is unramified, and so
factors through $F^{\times}_v/\OL^{\times}_v \simeq \Z$, and so $\phi(\Frob_v) \mod \m$ is
well defined;  assume that $\phi(\Frob_v) \not\equiv \pm 1 \mod \m$.
Then $\rho|D_v \sim  \phi \eps \oplus \phi^{-1}$.
\item Suppose that~$k = 2$, and that the level is prime to~$v|p$. Then~$\rho |D_v$ is finite flat. \label{part:finite}
\end{enumerate}
\end{conjectureOrd}

\begin{remark} \emph{If one
drops the assumption that~$T_v \not\in \m$ for~$v|p$ and
still assumes the corresponding version of assumption~\ref{part:finite}, one can also  
expect to prove a modularity lifting theorem in weight~$k = 2$ without an ordinary
hypothesis. However, it seems plausible that one might be able to prove  the weaker form of Conjecture~\ref{conj:A} without assuming 
the finite flatness condition.
If we drop this assumption, our arguments apply verbatim in all situations except when~$k = 2$ and~$\rhobar |_{D_v}$ for some~$v|p$ has the very special form
that it is finite flat but also admits non-crystalline semistable lifts. One may even be able to handle this case as well by a trick using Hida families (see 
Remark~\ref{remark:hida}) but we do not attempt to fill in the details.}
\end{remark}

\subsection{Assumptions}

Let~$k$ be a finite field of characteristic~$p$.
We shall assume, from now one, that the representation:
$$\rhobar_{\m}: G_{F} \rightarrow \GL_2(k)$$
satisfies all the hypotheses of Theorem~\ref{theorem:main}. In particular, it has determinant~$\eps^{k-1}$,
the restriction~$\rhobar |_{F(\zeta_p)}$ is absolutely irreducible, and there exist suitable collections of Taylor--Wiles primes.

\subsection{Patched Modules} \label{section:patched}

Using the methods of~\cite{CG}, we may patch together for any~$T$ and~$R$  (and any non-Eisenstein~$\m$) the homology 
groups~$H_*(Y_1(Q_N;R)[1/T],\LL/\varpi^N)$ to obtain a complex~$P_{\infty}$ such that:
\begin{enumerate}
\item $P_{\infty}$ is a perfect complex of finite~$S_{\infty}$-modules supported in
degrees~$m+2$ to~$1$, where~$S_{\infty} = W(k) \llbracket x_1,\ldots,x_{q} \rrbracket$ is
the patched module of diamond operators,  
where~$q-1$ is the dimension of the minimal adjoint Selmer group~$H^1_{\emptyset}(F,\ad^0(\rhobar))$,
and~$q$ is the dimension of the minimal dual Selmer group~$H^1_{\emptyset^*}(F,\ad^0(\rhobar)(1))$.
\item Let~$\a = (x_1,\ldots,x_{q})$ be the augmentation ideal of~$S_{\infty}$, and let~$\a_N
= ((1+x_1)^{p^n} - 1, \ldots, (1 + x_{q})^{p^n} - 1)$ be the ideal with~$S_{\infty}/\a_N = \Z_p[(\Z/p^n \Z)^q]$. Then
 $$H_*(P_{\infty} \otimes S_{\infty}/(\a_N,\varpi^N)) = H_*(Y_H(Q_N;R)[1/T],\LL/\varpi^N)_{\m}$$
for infinitely many sets of suitable Taylor-Wiles primes~$Q_N$ which are~$1 \mod p^n$, and~$Y_H$
is the quotient of~$Y_1$ which is a cover~$Y_0$ with Galois group~$\Delta = (\Z/p^n \Z)^{q}$. Moreover,
$$H_*(P_{\infty}) = \projlim H_*(Y_H(Q_N;R)[1/T],\LL/\varpi^N)_{\m}.$$
We denote these patched homology groups by~$\Htw_*(Y_0(R)[1/T],\LL)$.
\end{enumerate}
Note that we can do this construction with the addition of some auxiliary level structure, and also
simultaneously for any finite set of different auxiliary level structures.

\section{The Galois action in low degrees}

Let~$\t_{Q,\m}$ be the quotient of~$\T_{Q,\m}$ which acts faithfully in degrees~$ \le m$, namely
on $$\bigoplus_k \bigoplus_{i \le m} H_i(Y_1(Q_N;R)[1/S],\LL/\varpi^k)_{\m}.$$

\begin{prop} \label{prop:st} There exists an integer~$k$ depending only on~$m  = |S|$ such that
there exists a representation
$$\rho^{\t}: G_{F} \rightarrow \GL_2(\t_{Q,\m}/I)$$
where~$I^k = 0$ and such that $\rho^{\t}$ is Steinberg or unramified quadratic twist of Steinberg at primes dividing~$S$.
\end{prop}

\begin{proof}
We proceed by induction. 
 Suppose that $S = vT$, where~$T$ has~$m-1$ prime divisors.
From the amalgam sequence of Remark~\ref{remark:amalgam}, we find that
 there is an exact sequence:
$$H_{n}(Y_1(Q_N;R)[1/T],\LL/\varpi^r)^2_{\m} \rightarrow H_n(Y_1(Q_N;R)[1/S],\LL/\varpi^r)_{\m}
\rightarrow H_{n-1}(Y_1(Q_N;Rv)[1/T],\LL/\varpi^r)_{\m}.$$
We have~$U^2_v - 1 = 0$ for $v|S$ on $H_*(Y_1(Q_N;R)[1/S],\LL/\varpi^r)$. It follows that, for
the Galois representation associated to the image of the LHS, the eigenvalues of $\Frob_v$ are precisely $N(v)^{k-1}$ and $N(v)^k$, or~$-N(v)^{k-1}$
and~$-N(v)^{k}$, depending only on~$\rhobar$ (note that~$p \ne 2$, so the eigenvalue of~$U_v \in \{ \pm 1\}$ is
determined by~$\rhobar$). Moreover, by induction, the Galois representation associated to the RHS is Steinberg at~$v$.
Hence, again after possibly increasing the ideal of nilpotence,  it follows
that the middle term also gives rise to a Steinberg representation.
\end{proof}

\medskip

The key part of the argument is to show that the action of Galois in low degrees is unramified ``up to a small error.''

Following~\cite{CG}, we may, by finding suitably many sequences of Taylor--Wiles primes, patch
all these homology groups (localized at~$\m$) for all time. (We need only work with a finite fixed
set of auxiliary level structures.) The corresponding patched modules will be, assuming local--global compatibility conjectures,  modules over a framed local deformation ring~$\Rloc$, which will be a power series over the tensor product of local framed deformation
rings~$R_v$ for~$v|RS$.
We choose the local deformation ring~$R_v$ for~$v|p$ to be the ordinary crystalline deformation ring.
This coincides with the ordinary deformation ring unless~$k = 2$ and the semi-simplification of~$\rhobar|D_v$ is a twist of~$\eps \oplus 1$.
In the former case, the ordinary deformation ring is irreducible. In the latter case, the additional finite flat condition also
means that~$R_v$ is irreducible.
The local deformation rings~$R_v$ for~$v|S$ have two components corresponding to the unramified and Steinberg
representations respectively, and two corresponding equi-dimensional quotients
$R^{\st}_v$ and~$R^{\ur}_v$. Their intersection~$R^{\st,\ur}_v$ is also equi-dimensional 
with~$\dim(R^{\st,\ur}_v) = \dim(R^{\st}_v) - 1 = \dim(R^{\ur}_v) - 1$.
The ring~$\Rloc$ correpsondingly has~$2^m$ quotients on which one chooses a component of~$R_v$ for~$v|S$.
The common quotient~$\Rloc^{\st,\ur}$
has dimension~$\dim(\Rloc) - m$.

\medskip

The patched modules~$ \Htw_{i}$ are also naturally modules over a patched ring of diamond operators~$S_{\infty}
= \OL \llbracket x_1, \ldots ,x_{q} \rrbracket$.
In the context of~\cite{CG}, we have~$\ell_0 = 1$, or that~$\dim(S_{\infty}) = \dim(\Rloc) - 1$.
 We  have an exact sequence as follows:
$$\ldots \Htw_{i}(Y_1(Q_N;Rv)[1/T]) \rightarrow \Htw_{i}(Y_1(Q_N;R)[1/T])^2 \rightarrow
\Htw_{i}(Y_1(Q_N;R)[1/S]) \rightarrow \Htw_{i-1}(Y_1(Q_N;Rv)[1/T]) \ldots.$$

For a finitely generated~$S_{\infty}$-module~$M$, let the co-dimension of~$M$ denote the co-dimension
of the support of~$M$ as an~$S_{\infty}$-module.

\begin{prop} \label{prop:keyestimate} Let~$S$ be divisible by~$m$ primes. We have the following estimate:
$$\codim_{S_{\infty}} \Htw_{i}(Y_0(R)[1/S]) \ge \begin{cases} m-i + 3 & i \le m \\ 1, & i = m+1. \end{cases}$$
\end{prop}

\begin{proof}
The claim  for~$i = m+1$ follows by considering dimensions of deformation rings, because these modules are finite over~$\Rloc$.
For For~$i \le m$, we proceed via induction on~$m$. Write~$S = vT$, where~$T$ has~$m-1$ prime
factors.
There is an exact sequence:
 $$ \Htw_{i}(Y_0(R)[1/T])^2 \rightarrow
\Htw_{i}(Y_0(R)[1/S]) \rightarrow \Htw_{i-1}(Y_0(vR)[1/T]).$$
Assuming that~$i \le m$, we have~$i-1 \le m-1$. 
In the Serre category of~$S_{\infty}$-modules
modulo those of
 co-dimension at least~$(m-1) - (i-1) + 3 = m -i + 3$, we therefore have a surjection:
 $$ \Htw_{i}(Y_0(R)[1/T])^2 \rightarrow
\Htw_{i}(Y_0(R)[1/S]).$$
This implies that the Galois representation associated to the latter module is, (in this category)  unramfied at~$v$; and, using other~$v$, 
for all~$v|S$. 
It suffices to show that the RHS is zero, or equivalently, that it does not have co-dimension at 
most~$m-i+2$. We would  like to claim that,
by Proposition~\ref{prop:st}, 
the action of~$\Rloc$ in these degrees factors through the quotient
$\Rloc^{\st}$. This is not precisely true, since Proposition~\ref{prop:st} only says the Galois representation
is Steinberg after taking the quotient by a nilpotent ideal. If~$M$ is an~$S_{\infty}$-module, then
the support of~$M/J$ for a nilpotent ideal~$J$ will be the same as the support of~$M$
(see also the discussion in~\S\ref{section:nilpotent}). Hence, passing
to a suitable quotient of~$\Htw_{i}$, we may assume the module acquires an action of~$\Rloc$
which factors through~$\Rloc^{\st}$.
 Yet by what we have just shown above, the corresponding Galois representations
are also unramified at~$v|S$, and so are quotients of~$\Rloc^{\st,\ur}$.
Since~$\dim(\Rloc^{\st,\ur}) = \dim(\Rloc) - m =  \dim(S_{\infty}) - m -1$, we deduce that~$\Htw_i$
has co-dimension at least
$$m+1 > m - i + 2$$
providing that~$i \ge 2$.  If~$i = 0$, the module is trivial, because~$\m$ is not Eisenstein and~$H_0$ is Eisenstein.
If~$i = 1$, we are done by the congruence subgroup property,  which also  implies that~$H_i$ vanishes
after localization at~$\m$.
\end{proof}

\section{Level Raising}

\subsection{Ihara's Lemma and the level raising spectral sequence}

We recall some required constructions and results from~\cite{CV}.
The following comes from Chapter~IV of~\cite{CV}.
Let $S = T \cup \{v\}$. 
Let~$\LL$ be a local system (which could be torsion). We assume that~$\LL/\varpi$ is self-dual. For example, we could take~$\LL = \OL/\varpi^k$ for some~$k$.

Let~$Y = Y(K)$ for some~$K$ of level prime to~$S$.
Let~$\m$
be a maximal ideal of~$\T$.  

\begin{lemma}[Ihara's Lemma] 
If $\m$ is not Eisenstein, then
$$H_1(Y_0(v)[1/T],\LL)_{\m} \rightarrow H_1(Y[1/T],\LL)^2_{\m}$$
is surjective.
\end{lemma}

\begin{proof}  It suffices to show that~$H_1(Y[1/S],\LL)_{\m}$ for~$S = Tv$
 is trivial. From the amalgam sequence~\ref{remark:amalgam}, we see the cokernel is  a quotient of the group
 $H_1(Y[1/S],\LL)_{\m}$, and hence it suffices to show that this is trivial.
The homology of~$Y[1/S]$ can be written as the direct sum of the homologies of~$S$-arithmetic groups commensurable
with~$\GL_2(\OL_F[1/S])$, and, by~\cite{CSP}, these groups satisfy the congruence subgroup property
(this crucially uses the fact that~$S$ is divisible
by at least one prime~$v$, and that the lattice is non-cocompact).  The congruence kernel has order dividing the group
of roots of unity~$\mu_F$. Since~$p > 2$ is unramfied in~$F$, this is trivial after tensoring with~$\Z_p$. An easy computation then shows that
the relevant  cohomology group is Eisenstein.  (See~\cite{CV},~\S~4.) \end{proof}

\medskip

In order to prove the required level raising result (Theorem~\ref{theorem:raising}), we also
need the level raising spectral sequence of~\cite{CV} (Theorem~4.4.1).
If $\m$ is non-Eisenstein, then the $E^1$-page of the spectral sequence is:
\begin{equation} \label{eq:spectral}
\begin{diagram}
0   & \lTo & \ldots & \lTo & \qquad 0 \qquad &\lTo & \qquad 0 \qquad  \\
H_2(Y,\LL)^{2^{|S|}}_{\m}  & \lTo & \ldots & \lTo & \bigoplus_{v|S} H_2(Y_0(S/v),\LL)^2_{\m}
& \lTo & H_2(Y_0(S),\LL)_{\m} \\
H_1(Y,\LL)^{2^{|S|}}_{\m}  & \lTo & \ldots & \lTo & \bigoplus_{v|S} H_1(Y_0(S/v),\LL)^2_{\m}
& \lTo & H_1(Y_0(S),\LL)_{\m} \\
0   & \lTo & \ldots & \lTo & \qquad 0 \qquad &\lTo & \qquad 0 \qquad  \\
\end{diagram}
\end{equation}
The vanishing of the zeroeth  and third row follow from the assumption that~$\m$ is not Eisenstein.
This spectral sequence converges to $H_*(Y[1/S],\LL)_{\m}$. Tautologically, it degenerates
on the~$E^2$-page. After tensoring with~$\Q$, the sequences above are exact at  all but the final term, corresponding to the fact
that $H_*(Y[1/S],\LL)_{\m} \otimes \Q$ vanishes outside degrees $[m+1,m+2]$.

\medskip

We now establish a level--raising result.

\begin{theorem} \label{theorem:raising} Let $\m$ be a non-Eisenstein maximal ideal  of $\T$ with
 residue field~$k$ of characteristic~$p$.
 Let~$S$ be a product of~$m$ primes $v$ so that 
$T^2_v - (1 + N(v))^2 \in \m$.
Then 
 $$H_*(Y,\LL/\varpi)_{\m} \ne 0 \Rightarrow H_{*}(Y[1/S],\LL/\varpi)_{\m} \ne 0.$$
 \end{theorem}

 \begin{proof} 
 Consider the spectral sequence of~\cite{CV} in equation~\ref{eq:spectral} above.
 It is clear that the upper right hand corner term remains unchanged after one
reaches the~$E^2$-page. Assuming, for the sake of contradition, that 
$H_{m+2}(Y[1/S],\LL/\varpi)_{\m}$ vanishes, it follows that
 the map
$$H_2(Y_0(S),\LL/\varpi)^2_{\m} \rightarrow \bigoplus_{v|S} H_2(Y_0(S/v),\LL/\varpi)_{\m}$$
is injective. By Poincar\'{e} duality, there is an isomorphism
$H_2(Y,\LL/\varpi)_{\m} \simeq H^1_c(Y,\LL/\varpi)_{\m}$. Here we use the fact that~$\LL/\varpi$ is a self-dual local system.
 Because~$\m$ is non-Eisenstein,
 there
is an isomorphism between $H^1_c(Y,\LL/\varpi)_{\m}$ and $H^1(Y,\LL/\varpi)_{\m}$.
Finally, by the universal coefficient theorem, $H^1(Y,\LL/\varpi)_{\m}$ is dual to $H_1(Y,\LL/\varpi)_{\m}$.
Hence taking the dual of the injection above yields  the surjection:
$$\bigoplus_{v|S} H_1(Y_0(S/v),\LL/\varpi)^2_{\m} \rightarrow H_1(Y_0(S),\LL/\varpi)_{\m}.$$
It suffices to show that this results in a contradiction.
By Ihara's lemma, it follows that the composite map
$$\bigoplus_{v|S} H_1(Y_0(S/v),\LL/\varpi)^2_{\m} \rightarrow H_1(Y_0(S),\LL/\varpi)_{\m}
\rightarrow H_1(Y,\LL/\varpi)^{2^{m}}_{\m}$$
is also surjective. 
Our assumption is that, for some choice of signs, the elements
$D_v = T_v \pm (1 + N(v)) \in \m$ for all $v | S$.  
The map above decomposes into a sum of maps from each individual term, each of
which factor as follows
$$H_1(Y_0(S/v),\LL/\varpi)^2_{\m} \rightarrow H_1(Y_0(S),\LL/\varpi)_{\m}
\rightarrow H_1(Y_0(S/v),\LL/\varpi)^2_{\m}
\rightarrow H_1(Y,\LL/\varpi)^{2^{m}}_{\m}$$
An alternative description of this map can be given by replacing every
pair of groups by a single term, and replacing the two natural degeneracy maps
with either the sum or difference of these maps (depending on a sequence
of choice of Fricke involutions, which depend on the sign occurring in~$D_v$)
we end up with a map of the form:
$$H_1(Y_0(S/v),\LL/\varpi)_{\m} \rightarrow H_1(Y_0(S),\LL/\varpi)_{\m}
\rightarrow H_1(Y_0(S/v),\LL/\varpi)_{\m}
\rightarrow H_1(Y,\LL/\varpi)_{\m}.$$
On the other hand, the composite of the first two maps is the map obtained by pushing
forward and then pulling back, which (after either adding or subtracting the relevant maps)
is exactly the Hecke operator~$D_v$.
It follows that the composite of the entire map is then killed if one passes
to the quotient $H_1(Y,\LL/\varpi)_{\m}/D_v H_1(Y,\LL/\varpi)_{\m}$.
In particular, it follows that the composite
$$\bigoplus_{v|S} H_1(Y_0(S/v),\LL/\varpi)_{\m} \rightarrow H_1(Y_0(S),\LL/\varpi)_{\m}
\rightarrow H_1(Y,\LL/\varpi)_{\m} \rightarrow H_1(Y,\LL/\varpi)_{\m}/I H_1(Y,\LL/\varpi)_{\m}$$
is zero, where $I$ is the ideal generated by~$D_v$ for all $v|S$.
This contradicts the surjectivity unless $I$ generates the unit ideal.
But this in turn  contradicts the assumption that~$D_v \in \m$ for all~$\m$.
\end{proof}

\section{The argument}
\label{section:argument}
Let~$\rho$ be as in Theorem~\ref{theorem:main}. By assumption, we have~$H_2(Y_0(S),\LL)_{\m} \ne 0$, by
the assumption that~$\rhobar$ is modular. 
Hence~$H_{m+2}(Y[1/S],\LL/\varpi)_{\m}$ is modular by  Theorem~\ref{theorem:raising}.
As in~\ref{section:patched}, we obtain a complex~$P_{\infty}$ such that:
\begin{enumerate}
\item $P_{\infty}$ is a perfect complex of finite~$S_{\infty}$-modules supported in
degrees~$m+2$ to~$1$.
\item $H_*(P_{\infty} \otimes S_{\infty}/(\a_N,\varpi^N)) = H_*(Y_1(Q_N)[1/S],\LL/\varpi^N)_{\m}$
for infinitely many sets of suitable Taylor-Wiles primes~$Q_N$, and moreover,
$$H_*(P_{\infty}) =: \Htw_*(Y[1/S],\LL) =  \projlim H_*(Y_1(Q_N)[1/S],\LL/\varpi^N)_{\m}.$$
\end{enumerate}
Suppose that the corresponding quotients had actions of Galois representations mapping to the entire Hecke rings~$\T$
rather than~$\T/I$ for some nilpotent ideals~$I$ of fixed order. Then this action would extend to an action of~$\Rloc^{\st}$ on
$H_*(P_{\infty})$, where here~$\Rloc^{\st}$ is defined to have Steinberg conditions at all primes in~$S$, an ordinary
condition at~$v|p$, and unramified elsewhere. In the special case when~$k = 2$
and one of
the representations~$\rhobar |_{D_v}$  for~$v|p$  is twist equivalent to a representation of the form
$$\left( \begin{matrix} \eps &  * \\ 0 & 1 \end{matrix} \right)$$
where~$*$ is peu ramifi\'{e}e, we take the local deformation ring at~$v|p$ to be the finite flat deformation ring instead.
The corresponding ring~$\Rloc^{\st}$ is reduced  of dimension~$\dim(S_{\infty}) - 1$ and has one geometric component.
We would be done as long as the co-dimension of~$H_*(P_{\infty})$ is  equal to at most one, because then
the action of~$\Rloc$ must be  faithful, and we deduce our modularity theorem.
  As in Lemma~6.2 of~\cite{CG}, if~$H_*(P_{\infty})$ has co-dimension at least~$2$, then it must be the case that, for some~$j \ge 2$,
 $$ \codim_{S_{\infty}} \Htw_{m+2 - j}(Y[1/S]) \le j,$$
 or, with~$i = m+2-j$ and~$i \le m$,
 $$ \codim_{S_{\infty}} \Htw_{i}(Y[1/S]) \le m +2-i.$$
  Yet this exactly contradicts 
 Proposition~\ref{prop:keyestimate} (with~$R = 1$), and we are done.

 \begin{remark} \label{remark:hida} \emph{
 If one wants to weaken Conjecture~\ref{conj:A} by omitting part~\ref{part:finite}, then one can instead
 work over the full Hida family, where the corresponding ordinary deformation ring once more has a single component (in every case). 
The modularity method in the Hida family case works in essentially the same manner, see~\cite{KT}, so one expects
that the arguments of this paper  can be modified to handle this case as well.
 }
 \end{remark}
 
 \subsection{Nilpotence}
 In practice, we only have an Galois representation to~$\T/I$ for certain nilpotent ideals~$I$.
 Equivalently, we only have a Galois representation associated to the action of~$\T$ on
 $$H_*(Y_1(Q_N)[1/S],\LL/\varpi^N)/I$$
 for ideals~$I$ with some fixed nilpotence. Even if there is no such ideal~$I$ when~$S = 1$,
our inductive arguments for higher~$S$ use exact sequences which increases the nilpotence.
 
 It suffices to show that~$\Rloc$ maps to the action of~$\T$ on
 certain sub-quotients of~$\Htw_{i}(Y[1/S])$ which are ``just as large'' as the modules~$\Htw_{i}(Y[1/S])$ themselves.
 Roughy, the idea is that one can also patch the ideals~$I$ to obtain an action of~$\T$ and~$S_{\infty}$ on
 $\Htw_{i}(Y[1/S])/I$ for some ideal~$I$ of~$S_{\infty}$ with~$I^k = 0$ and~$k$ depending only on~$S$ and~$\rhobar$.
 The Galois deformation rings now give lower bounds for the co-dimension of the modules~$\Htw_{i}(Y[1/S])/I$. 
 Since~$I^k = 0$, these can be promoted to give the same lower bounds for the co-dimension of the modules~$\Htw_{i}(Y[1/S])$,
 and 
 then the argument above will go through unchanged. This is (essentially) what we now do.

\section{Notes on nilpotent Ideals}
\label{section:nilpotent}

\subsection{Passing to finite level}
\label{subsection:finite}

Let~$S = \OL \llbracket \Delta_{\infty} \rrbracket$.
If $I$ and~$J$ are ideals of~$S$, then~$\Tor^i(S/I,S/J)$ is an~$S/I$ and a~$S/J$-module, hence an~$S/(I + J)$-module.
So, if~$\Tor^0(S/I,S/J) = S/(I + J)$ is finite, then so is~$\Tor^i(S/I,S/J)$. Hence, by induction, if~$M$ is finitely generated and~$\Tor^0(S/I,M)$
is finite, then so is~$\Tor^i(S/I,M)$.
Moreover, there is a spectral sequence:
$$\Tor^j(S_{\infty}/\a,H^i(P_{\infty})) \Rightarrow H^{i+j}(P_{\infty} \otimes_{S_{\infty}} S_{\infty}/\a )$$

\subsection{The setup}

Let~$\Delta_{\infty} = \Z^q_p$ and~$\Delta_N = (\Z/p^N \Z)^q$.
Let $S_{\infty} = \OL \llbracket \Delta_{\infty} \rrbracket$, and let~$S_N = \OL[\Delta_N]$.
We begin with the assumption that we have arranged things so that the complexes patch on the level of~$S_{\infty}$-modules.
That is, we have a complex~$P_{\infty}$ of finite free~$S_{\infty}$-modules  so that, if
$$P_N = P_{\infty} \otimes_{S_{\infty}} S_N/\varpi^N,$$
then~$H^*(P_N)$ is the complex of cohomology  associated to (infinitely many) Taylor--Wiles sets~$Q_N$ with coefficients in~$\OL/\varpi^N$.
There is a natural identification
$$H^*(P_{\infty})  = \projlim H^*(P_N),$$
and a natural map
$$H^*(P_{\infty}) \otimes_{S_{\infty}} S_N/\varpi^N \rightarrow  H^*(P_N).$$
Because everything is finitely generated, and so in particular $H^*(P_{\infty}) \otimes_{S_{\infty}}  S_N/\varpi^N$ is finite, there exists some function~$f(N)$ (which we may take to
be increasing and~$\ge N$) such that
$$H^*(P_{\infty}) \otimes_{S_{\infty}}  S_N/\varpi^N =  H^*(P_{f(N)}) \otimes S_N/\varpi^N.$$
Having fixed such a function~$f(N)$, we define~$A_N$ to be~$H^*(P_{f(N)}) \otimes S_N/\varpi^N$. By construction, there
is a natural \emph{surjective} map
$$A_N \rightarrow A_N \otimes S_M/\varpi^M \rightarrow A_M$$
for all~$N \ge M$, and $\projlim A_N = H^*(P_N)$. For various choices of~$Q = Q_N$ giving rise to$A_N$ (really the primes in~$Q_N$ are~$1 \mod p^{f(N)}$),
we get different actions of different Hecke algebras~$\T$. We shall construct quotients~$B_N$ of~$A_N$ on which~$R_{\infty}$ acts on the corresponding quotients of~$\T$
which act faithfully on~$B_N$, and then
patch to get a quotient~$B_{\infty}$ of~$A_{\infty} = H^*(P_{\infty})$ on which~$R_{\infty}$ also acts. The main point is to ensure that~$B_{\infty}$
has the same co-dimension as~$A_{\infty}$.

\subsection{Hecke Algebras}

For each~$Q = Q_N$, let~$\Delta = \Delta_N$. Letting~$\Phi$ run over all the quotients of~$\Delta$, and letting~$k$ run over all integers at most~$N$,
we shall define~$\T$ to be the ring of endomorphisms generated by Hecke operators on
$$\bigoplus_{\Phi,k} H^*(Y_1(\Phi),\LL/\varpi^k).$$
Localize at a non-Eisenstein ideal~$\m$. On each particular module~$A$ in the direct sum above
there is a quotient~$\T_A$ on which there exists a Galois representation with image in~$\GL_2(\T_A/I_A)$
where~$I^m_A =  0$ for some universally fixed~$m$. Note that:
\begin{enumerate}
\item One initially knows that~$m$ is bounded universally for any fixed piece~$A$. However, there is no problem taking direct sums. The point is as follows;
given rings~$A$ and~$B$ with ideals~$I_A$ and~$I_B$ such  that~$I^m_A = I^m_B = 0$, the ideal~$(I_A \oplus I_B)$ of~$A \oplus B$ satisfies~$(I_A \oplus I_B)^m = 0$.
In particular, if~$\T_{\Phi}$ is the quotient for a particular~$\Phi$ and the corresponding ideal is~$I_{\Phi}$, there is a map:
$$\T  \hookrightarrow  \left(  \bigoplus \T_{\Phi} \right)/\bigoplus I_{\Phi},$$
and hence the image is~$\T/I$ where~$I^m \subset \left(\bigoplus I_{\Phi} \right)^m = 0$.
\item If there exists a pseudo-representation to~$\T/I$ and~$\T/J$ there exists one to $\T/I \oplus \T/J$, and the image will be~$\T/(I \cap J)$.
Hence there exists a minimal such ideal~$I$.
\item If~$N \ge M$, there is a surjective map from~$\T_{Q_N} \rightarrow \T_{Q_M}$, where the sets~$Q_N$ and~$Q_M$ are compatible (that is, come from the
same set of primes). The reason this is surjective is that we are including all the quotients of~$\Delta$ in the definition of~$\T$. Again, by patching, the map~$\T_{Q_N} \rightarrow \T_{Q_M}/I_M$
has a Galois representation satisfying local--global, so it factors through a surjection
$\T_{Q,N}/I_N \rightarrow \T_{Q,M}/I_M$.
\end{enumerate}

In particular, for~$A_N$ and~$A_M$ drawn from the same set~$Q$, there is a commutative diagram
$$
\begin{diagram}
A_N & \rTo & A_N/I_N =: B_N\\
 \dOnto & & \dOnto \\
 A_M & \rTo & A_M/I_M =: B_M
 \end{diagram}
 $$
 The point of this construction is that~$B_N$ and~$B_M$ have actions of the Galois deformation rings~$R_Q$,
 and hence have actions of~$R_{\infty}$.
 Moreover, these actions are compatible in the expected way with the action of~$S_{\infty}$ as diamond operators and
 local ramification operators respectively.
 
 \begin{lemma}  \label{lemma:boot} Suppose that~$I$ is an ideal of  local ring~$(\T,\m)$ such that~$I^m = 0$,
 let $S \rightarrow \T$ be a ring homomorphism, let~$M$
 be a finitely generated~$\T$ and~$S$-module with commuting actions
 compatible with the map from~$S$ to~$\T$, and let~$J = \Ann_{S}(M/I M)$.
 Then~$J^m M = 0$.
 \end{lemma}
 
 \begin{proof} The module~$M$ has a filtration as~$\T$ and~$S$-modules with graded pieces~$I^k M/I^{k+1} M$ for~$k = 0$ to~$m - 1$.
 Hence it suffices to show that each of these graded pieces is annihilated by~$J$. 
However,
there is a  surjective  homomorphism of~$\T$ and~$S$ modules given by
$$\bigoplus_{I^k}  M/I M \rightarrow I^k M/I^{k+1} M,$$
where the sum goes over all generators~$g$ of~$I^k$ and sends~$M$ to~$gM \subset I^k M$.
Since~$J$ annihilates the source, it annihilates the target.
\end{proof}
 
 \medskip
 
For each~$N$, we now consider the extra data of a quotient~$B_N$ of~$A_N$ which carries an action of~$R_{\infty}$.
We patch to obtain a pair
$$H^*(P_{\infty}) = A_{\infty} \rightarrow B_{\infty},$$
where~$B_{\infty}$ has an action of~$R_{\infty}$ and~$S_{\infty}$, and there is a natural map~$S_{\infty} \rightarrow R_{\infty}$
which commutes with this action.
Let~$J = \Ann_{S_{\infty}}(B_{\infty})$.
We claim that~$J^m$ acts trivially on~$H^*(P_{\infty})$. To check this, it suffices to check this on~$A_N$ for each~$N$.
By construction, $A_N$ is a surjective system and hence so is~$B_N$. Thus~$B_{\infty}$ surjects onto~$B_N$, and hence~$J$
annihilates~$B_N$, and thus~$J^m$ annihilates~$A_N$ by Lemma~\ref{lemma:boot}.
Moreover, this same argument works term by term in each degree.

\begin{lemma} $\codim(B_{\infty}) = \codim(H^*(P_{\infty}))$ (in each degree) as an~$S_{\infty}$-module.
\end{lemma}

\begin{proof}
\label{lemma:bounds}
Let~$I = \Ann_{S_{\infty}}(H^*(P_{\infty}))$.
Because it is finitely generated, the co-dimension of~$H^*(P_{\infty})$ is the co-dimension of~$S_{\infty}/I$. Equally,
the co-dimension of~$B_{\infty}$ is the co-dimension of~$S_{\infty}/J$ (again using finite generation). 
Hence it suffices to show that
$$J^m \subset I \subset J \Rightarrow \codim(S_{\infty}/I) =  \codim(S_{\infty}/J).$$
One inequality is obvious. However, the former module has a finite filtration by~$J^k/J^{k+1}$,
which is finitely generated and annihilated by~$k$. \end{proof}

Note that this argument also applies to a submodule~$A'_{\infty} \subset A_{\infty}$.

\begin{remark} \emph{One way to view the lemma above is follows.
The co-dimension of a finitely generated module is defined in terms of the 
dimension of the support. The dimension of a closed subscheme of~$S_{\infty}$, on the other hand,
only depends on its reduced structure. }
\end{remark}

\medskip

Let us now remark how to modify the argument of section~\ref{section:argument}. All bounds on the co-dimensions of~$H^*(P_{\infty})$
still apply by combining the bounds on the appropriate deformation rings with Lemma~\ref{lemma:bounds}.
Hence we deduce that~$\Htw_{m+1}(Y[1/S],\LL)_{\m}$ has co-dimension one and thus (because~$\Rloc$ is reduced has only one
geometric component) is nearly faithful as an~$\Rloc$ module. From this we want to deduce that~$H_{m+1}(Y[1/S],\LL)_{\m}$
is also nearly faithful as an~$R$ module. The module~$H_{m+1}(Y[1/S],\LL)_{\m}$ differs from~$\Htw_{m+1}(Y[1/S],\LL)_{\m}/\a$
by other terms arising from the spectral sequence in~\ref{subsection:finite}. However, all those terms must be finite --- if not,
then there must be a smallest degreee~$j$ such that~$\Htw_{j}(Y[1/S],\LL)_{\m}/\a$ is infinite, which from the spectral sequence
will contribute something non-zero to~$\Htw_j(Y[1/S],\LL)_{\m} \otimes \Q$, an impossibility for~$j \le m$. Hence we 
obtain an isomorphism
$$R^{\red} = \T^{\red},$$
as required.

\section{Acknowledgements}

I would like to thank Toby Gee for some helpful remarks and corrections on an earlier version of this manuscript.
The debt this paper owes to the author's previous collaborations~\cite{CG,CV} with Geraghty and Venkatesh
should also be clear.

\bibliographystyle{amsalpha}
\bibliography{semistable}
 
 \end{document}